\documentclass[a4paper,11pt,reqno]{article}

\usepackage[utf8]{inputenc}
\usepackage[T1]{fontenc}
\usepackage{lmodern}
\DeclareFontFamily{OMX}{lmex}{}
\DeclareFontShape{OMX}{lmex}{m}{n}{%
   <-> lmex10
   }{}
\usepackage[english]{babel}
\usepackage{csquotes}
\usepackage{microtype}

\usepackage{enumitem}
\setlist[enumerate,1]{label=(\roman*), ref=\roman*}
\setlist[enumerate,2]{label=\alph*., ref=\theenumi.\alph*}
\setlist[enumerate,3]{label=\arabic*., ref=\theenumii.\arabic*}

\usepackage{tikz}
\definecolor{midnightblue}{rgb}{0.1, 0.1, 0.44}

\usepackage[backend=biber,hyperref=true,backref=true,style=alphabetic,firstinits=true]{biblatex}
\usepackage{biblatex-mr}
\usepackage[colorlinks,allcolors=midnightblue,bookmarks=true]{hyperref}
\AtBeginDocument{}
\hypersetup{
	pdftitle={Linear extensions through random projection},
	pdfauthor={Elia Bruè, Simone Di Marino, Federico Stra},
	pdfsubject={Extension of Lipschitz functions},
	pdfkeywords={Lipschitz, Whitney, extension, metric, capacity},
	bookmarksnumbered=true,
	bookmarksopen=true,
	bookmarksopenlevel=1
}

\usepackage[a4paper,hmargin=3.25cm,bottom=4cm,top=3.5cm,footskip=3\baselineskip]{geometry}

\usepackage{fsmath}

\begin{document}

\title{Linear Lipschitz and $C^1$ extension operators through random projection}
\author{
	Elia Bruè\thanks{
	Scuola Normale Superiore, Pisa, \url{elia.brue@sns.it}}
	\and
	Simone Di Marino\thanks{
	Indam, Unità SNS, Pisa, \url{simone.dimarino@altamatematica.it}}
	\and
	Federico Stra\thanks{
	Scuola Normale Superiore, Pisa, \url{federico.stra@sns.it}}
}

\maketitle

	\begin{abstract}
	We construct a regular random projection of a metric space onto a closed doubling subset and use it to linearly extend Lipschitz and $C^1$ functions. This way we prove more directly a result by \textcite{lee-naor-05} and we generalize the $C^1$ extension theorem by \textcite{whitney} to Banach spaces.
\end{abstract}

\tableofcontents


\section{Introduction}
\label{sec:introduction}

The aim is to provide extension theorems $\Lip(X;Z)\to\Lip(Y;Z)$ where $X\subset Y$ is a closed subset of a complete metric space $(Y,\dist)$ and $Z$ is a Banach space, under hypotheses just on the space $X$ alone and not on the ambient space $Y$.

In \cite{lee-naor-05} the authors provide the following extension theorem for Lipschitz functions in a metric setting.

\begin{theorem*}[\textcite{lee-naor-05}]
	Let $X\subset(Y,\dist)$ be a doubling metric space with doubling constant $\lambda_X$. Then there is an extension $T:\Lip(X;Z)\to\Lip(Y;Z)$ such that
	\[
	\Lip(Tf) \leq C \log(\lambda_X) \Lip(f) \qquad \forall f\in\Lip(X;Z),
	\]
	where $C$ is a universal constant.
\end{theorem*}

Our goal is to obtain more directly the previous result, through a simpler proof based on ideas appearing in \cite{johnson-lindenstrauss-schechtman}. See also \cite{lee-naor-04,ohta} for related discussions. With this method, we can provide also a $C^1$ extension result in the spirit of Whitney \cite{whitney}.

In Section \ref{subsec:lipext} we also have a simple and very short proof of the Lipschitz extension result which is self-contained and based solely on the existence of a doubling measure on a doubling space.

The main theorems are \autoref{thm:lipschitz-extension} for the Lipschitz extension and \autoref{thm:whitney} for the $C^1$ extension respectively. The structure is as follows: in \autoref{sec:partitions} we contruct partitions of unity, both in the Lipschitz and $C^1$ version; in \autoref{sec:random-projections} we use these partitions to build Lipschitz and $C^1$ random projections of a space onto a subspace and finally in \autoref{sec:linear-extensions} we prove the extension theorems using the previously developed tools.

\subsection{Notation and preliminaries}

Let $(X,\dist)$ be a complete metric space. We will denote with $B(x,r)$ the open ball of radius $r$, centered at $x$ and, for $A \subseteq X$, we define $\dist(x, A)= \inf \set{ \dist(x,x') }{ x' \in A }$.

We will denote by $\Lip(X;Z)$ the set of Lipschitz functions with values in $Z$; if the second space is dropped it means that $Z=\setR$. Moreover, given $f \in \Lip(X;Z)$, we denote by $\Lip(f)$ the least Lipschitz constant for the function $f$.
We make use of the notion of slope of a function $f: X \to \setR$ defined as
\[
\abs{\nabla f} (x) = \limsup_{y \to x} \frac{ \abs{f(y)-f(x)}}{\dist(y,x)}.
\]

We will be dealing with measures supported in metric spaces: we denote by $\Prob(X)$ the set of Borel probability measures on $X$, with $\Meas_+(X)$ the set of finite nonnegative Borel measures on $X$ and with $\Meas(X;Z)$ the set of vector valued measures with finite total variation. As before, if the second spaces is omitted then $Z=\setR$ and so it will reduce to the space of signed measures.

Of crucial importance in the sequel will be the $W_1$ Wasserstein distance. We recall here just the dual representation instead of the direct one because it will be the more relevant for the further development. 
We will define it as usual on $\Prob_1(X)$, where
 $$\Prob_1(X)=\set*{\mu\in\Prob(X)}{\int_X\dist(x,x_0)\d\mu(x), \text{ for some $x_0\in X$}}.$$

\begin{definition}[Wasserstein distance]\label{def:w1} 
	Let $\mu_1, \mu_2 \in \Prob_1(X)$. Then we define
	\[
	W_1(\mu, \nu) = \sup \set*{
		\int_X f \d \mu - \int_X f \d \nu
	}{
		f \in \Lip(X),\ \Lip(f) \leq 1 }.
	\]
\end{definition}

\begin{remark}
	Notice that there is no harm in defining $W_1$ only on $\Prob_1(X)$, since in the sequel we will deal only with probabilities with bounded supports, which clearly belong to $\Prob_1(X)$. A useful inequality that follows directly from the definition is
	\begin{equation}\label{eqn:w1lip}
	\abs*{ \int_X f \d\mu_1 - \int_X f \d\mu_2 }
	\leq \Lip(f) W_1(\mu_1, \mu_2).
	\end{equation}
\end{remark}

Throughout the paper we use the notation $\lesssim$ to omit a universal constant not depending on $X$, $Y$, the doubling constant $\lambda$ or anything of this sort. We will use two notion of dimensionality of a metric space: the \emph{doubling constant} and the \emph{metric capacity}.

\begin{definition}[Doubling metric space]\label{def:doubling}
	$(X,\dist)$ is a doubling metric space if there exists $\lambda \in \mathbb{N}$ such that  every ball of radius $2r$ can be covered with at most $\lambda$ balls of radius $r$. The least such constant is $\lambda_X$, the doubling constant of $X$\footnote{In the sequel we will drop the dependence on $X$ when there is no room for confusion}.
\end{definition}

\begin{definition}[Metric capacity]\label{def:capacity}
	Given a metric space $(X,\dist)$ we define the \emph{metric capacity}\footnote{For short, just \emph{capacity} in the sequel.} $\kappa_X:(0,1]\to\setN\cup\{\infty\}$ as
	\[
	\kappa_X(\eps) = \sup\set*{k}
	{\exists x_0,\dotsc,x_k\in X,\ \exists r>0
		\text{ s.t. }
		\bigsqcup_{i=1}^k B(x_i,\eps r)\subset B(x_0,r)},
	\]
where the notation $\bigsqcup_{i=1}^k B(x_i,\eps r)$ indicate a \emph{disjoint} union of balls.

\end{definition}

It can be verified that if $\kappa_X(\eps)<\infty$ for some $\eps < 1/3$, then $\kappa_X(t)$ is finite for every $t \in (0,1]$. Even if it is true that $X$ has a finite doubling constant iff $X$ has a finite metric capacity, it is more natural to use the latter in some of the constructions. However since we want the final result to depend only on the doubling constant of $X$, we will make use of the following proposition comparing $\lambda$ and $\kappa$.

\begin{proposition}[Comparing $\kappa$ and $\lambda$]\label{prop:est}
	Let $X$ be a metric space. Then we have that
	\begin{enumerate}
		\item $\lambda \leq \kappa_X(1/5)$;
		\item $\kappa_X(\eps) \leq \lambda^k$ whenever $\frac 1{2^k} < \eps \leq \frac 1{2^{k-1}}$.
	\end{enumerate}
\end{proposition}

\begin{proof} Considering a maximal family $\mathcal{F}= \{B(x_i, \eps r )\}_{i \in I}$ of disjoint balls contained in $B(x_0,r)$ we have $\abs{\mathcal{F}}\leq \kappa(\eps)$ and moreover $B\bigl(x_0, (1-\eps) r\bigr) \subseteq \cup_i B(x_i, 2 \eps r)$. Choosing $\eps=1/5$ and thanks to the arbitrariness of $r$ and $x_0$ we get that $\lambda \leq \kappa_X(1/5)$.
	
	%
	%
	
	In order to prove the second inequality we first observe that for we can cover $B(x_0, 2^k r)$ with less than $\lambda^k$ balls of radius $r$: let us consider $\mathcal{F}' = \{ B(y_i, r)\}$ such a family. Let $\frac 1{2^k} < \eps \leq \frac 1{2^{k-1}}$ and  $\mathcal{F} = \{B(x_i, \eps 2^k r)\}$ be a disjoint family of balls  contained in $B(x_0,2^k r)$. It is now easy to see that $B(y_i,r)$ can contain at most one $x_i$; then we have $\abs{\mathcal{F}} \leq \abs{\mathcal{F}'} \leq \lambda^k$ and so $\kappa_X(\eps) \leq \lambda^k $.
	
	%
\end{proof}

\section{Whitney-type partitions}
\label{sec:partitions}

The way to the extension results follows the same path traced by Whitney for his theorem [add reference], with the addition of some ideas that we have learnt from \cite{johnson-lindenstrauss-schechtman}. The first step is to construct suitable partitions of unity so that manually built local extensions can be patched together at the global level. Since our goal is to prove Lipschitz and $C^1$ extendability, we are going to need two different kind of partitions, one for each purpose. The underlying ideas are the same in both cases; in particular, the attentive reader will notice that in the $C^1$ construction we try to replicate the proof of the Lipschitz version, with appropriate modifications.

\begin{proposition}[Relative Lipschitz partition of unity]\label{prop:lipschitz-partition}
	Let $(Y,\dist)$ be a metric space and $X\subset Y$ a closed subset with finite doubling constant $\lambda$. Then there exists a countable family $\{V_i,\phi_i,x_i\}_i$ such that:
	\begin{enumerate}
		\item $\{V_i\}_i$ is a locally finite covering of $Y\setminus X$ with covering constant $ 3 \lambda^4$;
		\item $\{\phi_i\}_i$ is a partition of unity on $Y\setminus X$ such that $\{\phi_i>0\}\subset V_i$ and
		\[
		\sum_{i} \abs{\nabla\phi_i}(y) \lesssim \frac{\log\lambda}{\dist(y,X)} ;
		\]
		\item the points $x_i$ belong to $X$ and $\dist(y,x_i)\lesssim\dist(y,X)$ if $y\in V_i$.
	\end{enumerate}
\end{proposition}

\begin{proof} This follows directly from \autoref{lem:lip_part} below, re-indexing the family $\{V_i^n, \phi_i^n, x_i^n\}_{i,n}$.
\end{proof}

The idea is that thanks to (iii) we have that $x_i$ is an approximate projection of any $y \in V_i$ on $X$ and in fact this partition of unity will help us define a \emph{random projection}. The estimate (ii) will be instead crucial to prove Lipschitz estimates. The next proposition will be used to prove an extension of Whitney theorem for Banach spaces, requiring the partition of unity to be $C^1$. Unfortunately the dependence of $\lambda$  in the estimates of the slopes is much worse in this case: it will be interesting to have a class of Banach spaces where we can recover the same logarithmic behavior as in the Lipschitz case.

\begin{proposition}[Relative $C^1$ partition of unity]\label{prop:C1-partition}
	Let $Y$ be a Banach space whose norm belongs to $C^1(Y\setminus\{0\})$ and let $X\subset Y$ be a closed subset with doubling constant $\lambda$. Then there exists a family $\{V_i,\phi_i,x_i\}_i$ such that:
	\begin{enumerate}
		\item\label{it:c1part-1} $\{V_i\}_i$ is a locally finite covering of $Y\setminus X$ with covering constant  $ 5 \lambda^4$;
		\item\label{it:c1part-2} $\{\phi_i\}_i$ is a partition of unity on $Y\setminus X$ such that $\{\phi_i>0\}\subset V_i$ and
		\[
		\sum_{i} \abs{\nabla\phi_i}(y) \lesssim \frac{\lambda^4 \log\lambda}{\dist(y,X)} ;
		\]
		moreover $\phi_i\in C^1(Y)$ for every $i\in\setN$.
		\item\label{it:c1part-3} the points $x_i$ belong to $X$ and $\dist(y,x_i)\lesssim\dist(y,X)$ if $y\in V_i$.
	\end{enumerate}
\end{proposition}

\begin{proof} This follows directly from \autoref{lem:c1_part}, taking the family $\{A_i^n, \phi_i^n, x_i^n\}_{i,n}$.
\end{proof}

%
%
%
%
We now state and prove a simple technical lemma, crucial in the construction of the Whitney-type covering in \autoref{lem:whitney-covering}.

\begin{lemma}\label{lem:covering}
	Let $(X,\dist)$ be a metric space. Then for every $r>0$ there exists a family of disjoint balls $\{(B_i=B(x_i,r)\}_{i\in I}$ such that $\{2B_i=B(x_i,2r)\}_{i\in I}$ is a covering of $X$.
\end{lemma}

\begin{proof}
	Let $\mathcal F=\set{(B_i)_{i\in I}}{B_i\cap B_j=\emptyset}$ be the collection of all disjoint families of open balls of radius $r$. A simple application of Zorn's lemma shows that there exist a maximal family $(B_i)_{i\in I}$. Suppose by contradiction that $x\not\in 2B_i$ for any $i\in I$. Then $B(x,r)$ is disjoint from every $B_i$, contradicting the maximality.
\end{proof}

\begin{lemma}[Whitney-type covering]\label{lem:whitney-covering}
	Let $(Y,\dist)$ be a complete metric space and $X\subset Y$ a closed subset with finite capacity. For every $n\in\setZ$ let $\{B_i^n=B(x_i^n,2^n)\}_{i\in I_n}$ be a family given by \autoref{lem:covering}. Let
	\[
	\tilde V_i^n = \set{y\in Y\setminus X}{
		2^{n} \leq  \dist(y,X) < 2^{n+1} \text{ and }
		\dist(y,x_i^n) = \min_{j\in I_n} \dist(y,x_j^n)}.
	\]
	Then the family of enlarged sets $\mathcal{F}=\set*{V_i^n=(\tilde V_i^n)_{2^{n-1}}}{n\in\setZ,i\in I_n}$ has the following properties:
	\begin{enumerate}
		\item $\mathcal{F}$ is a locally finite covering of $Y\setminus X$ with constant $3\kappa_X(1/10)$;
		\item for every $y \in Y \setminus X$ we have $ \dist (y,X)/4 \leq \max_{V \in \mathcal{F}}\{ \dist(y,V^c) \} \leq \dist(y,X)$.
	\end{enumerate}
\end{lemma}

\begin{proof}
	First of all, it is obvious that $\mathcal{F}$ is a covering: in fact also $\{\tilde{V}_i^n\}_{i,n}$ is a covering.
	Let us prove that for $y \in V_i^n$ we have $\dist(y,x_i^n) \leq  9 \cdot 2^{n-1}$. By definition, for every $\eps>0$ there exists $\tilde{y} \in \tilde{V}_i^n$ and $x \in X$ such that
	\[
	\dist(y,\tilde{y})< \dist(y,\tilde{V}^n_i) + \eps \leq 2^{n-1}+\eps
	\qquad \text{and}
	\qquad \dist(\tilde{y},x)< \dist(\tilde{y},X) + \eps \leq 2^{n+1} +\eps.
	\]
	Then, by the covering property of $\{2B_i^n\}_{i \in I_n}$ we know that there exists $j$ such that $x \in 2B_j^n$ and so $\dist (x, x_j^n) \leq 2^{n+1}$. In particular, by definition of $\tilde{V}_i^n$ we obtain
	\begin{align*} \dist(y,x_i^n) & \leq \dist( \tilde{y} , x_i^n) + \dist(\tilde{y}, y) \leq \dist(\tilde{y},x_j^n) +  \dist(\tilde{y}, y) \\ & \leq \dist(\tilde{y},x) + \dist (x, x_j^n) +  \dist(\tilde{y}, y) \leq 9 \cdot 2^{n-1}+2\eps.
	\end{align*}
	In order to get the local finiteness in (i) we use the fact that if $y \in V^n_i \cap V^n_j$ then we have $\dist (y,x_i^n) \leq 9 \cdot 2^{n-1}$ and $\dist (y,x_j^n) \leq 9 \cdot 2^{n-1}$. In particular we have $x_j^n \in B(x_i^n, 9 \cdot 2^n)$ and so $B(x_j^n , 2^n) \subseteq B(x_i^n, 10 \cdot 2^n )$. In particular we get that $ \sharp \{ j : y \in V_j^n \} \leq \kappa_X(1/10)$. Now, knowing that $y \in V_i^n$ implies $2^{n-1}<\dist (y,X)<2^{n+2}$ we have at most three possible choices for $n$ and at most $\kappa_X(1/10)$ sets for every $n$, so the conclusion.
	
	For (ii) the inequality $\max_{V \in \mathcal{F}}\{ \dist(y,V^c) \} \leq \dist(y,X)$ is trivial since $X \subset V^c$ for all $V$. For the other inequality we know that $y \in \tilde{V}^n_i$ for some $i,n$ and in particular we have $\dist (y, (V^n_i)^c) \geq 2^{n-1}$ by the definition of $V^n_i$. But then we have
	\[
	\frac {\dist(y,X)}4 < 2^{n-1} \leq \dist (y, (V^n_i)^c) \leq \max_{V \in \mathcal{F}}\{ \dist(y,V^c) \}. \qedhere
	\]
\end{proof}
In the final part of this section we build the two families of partitions of unity: the first one is made by Lipschitz functions (\autoref{lem:lip_part}) and the second, more regular, it is composed by $C^1(Y)$ functions (\autoref{lem:c1_part} ).

\begin{lemma}[Lipschitz partition of unity]\label{lem:lip_part}
	Let $\{V_i^n\}_{i,n}$ be the sets given by \autoref{lem:whitney-covering}. For $m>0$ define the functions
	\[
	\tilde\phi_i^n(y)=\dist^m\bigl(y,(V_i^n)^c\bigr) \quad\text{and}\quad
	\phi_i^n(y)=\frac{\tilde\phi_i^n(y)}{\sum\limits_{k,j} \tilde\phi_j^k(y)}.
	\]
	Then the family $\{\phi_i^n\}_i^n$ is a partition of unity with the property that
	\[
	\sum_{n,i} \abs{\nabla\phi_i^n}(y)
	\lesssim \frac{\log\lambda}{\dist(y,X)} .
	\]
\end{lemma}

\begin{proof}
	Thanks to the sublinearity of the slope, the chain rule, and the fact that $\abs{\nabla \dist (y, A)} \leq 1$ for every $A$, we obtain
	\[
	\abs{\nabla\phi_i^n}(y) \leq
	m \frac{ \dist^{m-1}\bigl(y,(V_i^n)^c\bigr) }
	{ \sum _{k,j} \dist^{m}\bigl(y,(V_j^k)^c \bigr) }
	+ m \frac { \dist^{m}\bigl(y,(V_i^n)^c\bigr) \cdot
		\sum _{k,j} \dist^{m-1}\bigl(y,(V_j^k)^c\bigr) }
	{\left( \sum _{k,j} \dist^{m}\bigl(y,(V_j^k)^c\bigr) \right)^2}.
	\]
	
	
	In order to have a clearer exposition, we fix $\{ \dist_l \}_{l \in \{1, \dotsc,  N\} } = \{\dist^{m-1}\bigl(y,(V_j^k)^c\bigr) \}_{j,k} $ where we included all couples $j,k$ such that $y \in V_j^k$; in particular we have $N \leq 2 \kappa_X (1/10 )$. Then summing up on the indices $i,n$ and simplifying we get
	\[
	\sum_{i,n} \abs{\nabla\phi_i^n}(y) \leq m \frac{ \sum_{l} \dist_l^{m-1} }{ \sum_{l} \dist_l^m }  + m \frac{ \sum_l \dist_l^m  \cdot  \sum_{l} \dist_l^{m-1} }{ (\sum_{l} \dist_l^m)^2 }  = 2m \frac{ \sum_{l} \dist_l^{m-1} }{ \sum_{l} \dist_l^m } .
	\]
	Now we use the inequality between the means $\Bigl(\frac{ \sum_l \dist_l^{m-1}}N \Bigr)^{1/(m-1)} \leq\Bigl(\frac{ \sum_l \dist_l^{m}}N \Bigr)^{1/m}$, obtaining
	\[
	\sum_{i,n} \abs{\nabla\phi_i^n}(y) \leq 2m  \frac {N^{1/m}}{\left( \sum_l \dist_l^m \right)^{1/m}}.
	\]
	By \autoref{lem:whitney-covering} (ii), we have $\max_l \{\dist_l\} \geq \dist (y, X)/4$ and so, using \autoref{prop:est} (ii) and then setting $m=\log_2 \lambda$ we find
	\[
	\abs{\nabla\phi_i^n}(y) \leq 2m  \frac { N^{1/m}}{\max_l \{\dist_l\}} \leq  \frac {8m \bigl(2 \kappa(1/10)\bigr)^{1/m}}{ \dist (y, X) } \leq 256  \frac{\log_2(\lambda)}{\dist(y,X)}.
	\]
\end{proof}

\begin{lemma}[$C^1$ partition of unity]\label{lem:c1_part}
	Let $X$ and $Y$ be as in \autoref{prop:C1-partition} and for every $n\in\setZ$ let $\{B_i^n=B(x_i^n,2^n)\}_{i\in I_n}$ be the family given by \autoref{lem:covering}. Then there exists a partition of unity $\{\phi_i^n\}_{i,n}$ of $Y\setminus X$ such that, denoting $A_i^n= \{\phi_i^n >0\}$, we have that
	\begin{enumerate}
		\item $ \{A_i^n\}_{i,n}$ is a covering of $Y \setminus X$ with  covering constant less than $C \lambda^6$;
		\item if $y \in A_i^n$ then $\dist(y,x_i^n) \lesssim \dist(y,X)$;
		\item $\sum_{i,n} \abs{\nabla\phi_i^n}(y)  \lesssim \frac{ \lambda^5 \log \lambda} {\dist (y,X)}$.
	\end{enumerate}
\end{lemma}

\begin{proof}
	The idea is to take
	\[
	\phi_i^n(y)=\frac{\tilde\phi_i^n(y)}{\sum\limits_{k,j} \tilde\phi_j^k(y)},
	\]
	where
	\[
	\tilde\phi_i^n(y) =
	\xi \oleft( 8\ell - \frac{\abs{x_i^n - y}}{2^n} \right) \cdot
	\xi \oleft( \frac{\abs{x_i^n - y}}{2^n}  - \ell \right)  \cdot
	\prod_{x_j^n \sim x_i^n} \xi\oleft(
	\frac{\abs{x_j^n - y}}{2^n} - \frac{\abs{x_i^n - y}}{2^n}
	+ \delta \right)
	\]
	and
	\begin{itemize}
		\item $\delta\ll1\ll\ell$;
		\item $\xi:\setR\to[0,1]$ is a suitably chosen increasing $C^1$ function satysfying $\xi(t)=0$ for $t\leq0$ and $\xi(t)=1$ for $t\geq\delta$,
		\item $\xi'\leq f(\xi)$ for a positive concave function $f:[0,\infty)\to[0,\infty)$ with $f(0)=0$ to be specified later,
		\item the notation $x_j^n \sim x_i^n$ means that $\abs{x_j^n-x_i^n} \leq 2^n(9\ell-\delta)$.
	\end{itemize}
	
	Fix $N = \kappa_X\bigl(1/(9\ell-\delta+1)\bigr)$. We will prove the lemma through the following steps.
	\begin{enumerate}[label=(\alph*)]
		\item $\tilde\phi_i^n(y) >0$ implies $(\ell-2-\delta) 2^n \leq \dist(y,X) \leq 8\ell \cdot 2^n$.
		\item Let $I_n(y) = \set{i\in I_n}{\tilde\phi_i^n(y)>0}$, then  $\abs{I_n(y)} \leq \kappa_X\bigl(1/(16\ell+1)\bigr)$. Moreover $\abs{\set{n}{I_n(y) \neq \emptyset}} \leq \left\lfloor \log_2 \bigl(8\ell/(\ell-2-\delta)\bigr) \right\rfloor$.
		\item $\abs{\nabla \tilde\phi_i^n}(y) \leq 2(N+1)2^{-n} f\bigl(\tilde \phi_i^n (y)\bigr)$.
		\item If $2\ell \leq \dist(y,X)2^{-n} \leq 4\ell$ then there exists $i \in I_n$ such that
		\[
		\tilde\phi_i^n(y) \geq \xi (4\ell-2) \xi(\ell) \xi( \delta )^N = 1,
		\]
		so that in particular $\sum_{i,n}\tilde\phi_i^n(y) \geq 1$ for every $y\in Y\setminus X$.
	\end{enumerate}
	
	We start by proving (a). It is obvious that $\dist(y,X) \leq \abs{y-x_i^n} \leq 8\ell\cdot2^n$. For the other inequality suppose by contradiction that there exists $x \in X$ such that $\abs{y-x} < 2^n (\ell-2-\delta)$; then there exists $j$ such that $\abs{x-x_j^n} \leq 2^{n+1}$ and so, by triangle inequality we have $\abs{x_j^n - x_i^n} \leq 2^n (9\ell-\delta)$ and in particular $x_j^n \sim x_i^n$. Then, using $\phi_i^n(y)>0$ we get
	\[
	\abs{y-x_j^n} \geq \abs{y-x_i^n} - \delta 2^n \geq (\ell-\delta) 2^n,
	\]
	which is in contradiction with
	\[
	\abs{y-x_j^n} \leq \abs{y-x} + \abs{x-x_j^n} < (\ell-\delta) 2^n.
	\]
	
	In order to prove (b) we fix $i \in I_n(y)$ and observe that for all $j \in I_n(y)$ we have $\abs{x_j^n - y} \leq 8\ell\cdot2^n$, and in particular $\abs{x_j^n - x_i^n} \leq 8\ell\cdot2^{n+1}$, so that $B(x_j^n ,2^n) \subseteq B\bigl(x_i^n , (16\ell+1)2^n\bigr)$ and thus the conclusion follows using the definition of $\kappa_X$. For the second cardinality computation, assume that $y\in A_i^{n_1}\cap A_j^{n_2}$; then from (a) we deduce $\abs{n_1-n_2}\leq\log_2 \bigl(8\ell/(\ell-2-\delta)\bigr)$.
	
	For (c) it is sufficient to use the chain rule, the fact that the distance to a fixed point is $1$-Lipschitz and that $f(a)b \leq f(ab)$ for $a,b \leq 1$ because of the concavity.
	
	The last point follows from taking $i \in I_n$ that minimizes $\abs{y - x_i^n}$. In this way we have that all the factors in the last product are always bigger than $\xi(\delta)$. As for the first two factor, for sure we have $\abs{y- x_i^n} \geq 2\ell\cdot2^n$ and, calling $\bar{y}$ a projection of $y$ on $X$, there exists $j$ such that $\abs{x_j^n -\bar{y}}\leq 2^{n+1}$. By the minimality of $i$ we get
	\[
	\abs{y-x_i^n} \leq \abs{y-x_j^n} \leq
	\abs{y-\bar{y}} + \abs{x_j^n - \bar{y}} \leq 2^n (4\ell+2).
	\]
	These two inequalities let us conclude.
	
	We now compute $\abs{\nabla \phi_i^n}$. Setting $K(y) = \abs{\set{(j,n)}{\phi_j^n (y) >0}}$, from (b) we deduce that
	\begin{equation}\label{eq:aaa}
	K(y) \leq \kappa_X\bigl(1/(16\ell+1)\bigr)
	\left\lfloor \log_2 \bigl(8\ell/(\ell-2-\delta)\bigr) \right\rfloor,
	\end{equation}
	which implies $(i)$. Now, using (c) we get
	\begin{align*}
		\abs{\nabla \phi_i^n}(y) & \leq \frac {2(N+1)}{2^n} \cdot \left(  \frac { f\bigl( \tilde \phi_i^n (y) \bigr)  } { \sum_{j,k} \tilde {\phi} _j^k (y)}  + \frac {  \tilde \phi_i^n (y)  } { \sum_{j,k} \tilde {\phi} _j^k (y)} \cdot \frac { \sum_{j,k} f\bigl( \tilde \phi_j^k (y) \bigr)  } { \sum_{j,k} \tilde {\phi} _j^k (y)} \right) \\
		\sum_{i,n} \abs{\nabla \phi_i^n} (y)
		& \leq \frac {2(N+1)}{2^n} \cdot \left(  \frac { \sum_{j,k} f\bigl( \tilde \phi_j^k (y) \bigr)  } { \sum_{j,k} \tilde {\phi} _j^k (y)}  + \frac {  \sum_{j,k} \tilde \phi_j^k (y)  } { \sum_{j,k} \tilde {\phi} _j^k (y)} \cdot \frac { \sum_{j,k} f\bigl( \tilde \phi_j^k (y) \bigr)  } { \sum_{j,k} \tilde {\phi} _j^k (y)} \right) \\
		& =  \frac {4(N+1)}{2^n} \cdot    \frac { \frac 1{K(y)} \sum_{j,k} f\bigl( \tilde \phi_j^k (y) \bigr)  } { \frac 1{K(y)} \sum_{j,k} \tilde{\phi} _j^k (y)}
		\leq \frac {4(N+1)}{2^n} \cdot   \frac { f\oleft( \frac 1{K(y)} \sum_{j,k} \tilde{\phi}_j^k (y) \right)  } { \frac 1{K(y)} \sum_{j,k} \tilde {\phi} _j^k (y)} \\
		&  \leq \frac{4(N+1)}{2^n} \cdot K(y) f \oleft( \frac1{K(y)} \right) \\
		&  \leq \frac1{\dist(y,X)}\cdot [32\ell(N+1)] K(y) f \oleft( \frac1{K(y)} \right) ,
	\end{align*}
	where we used the concavity of $f$, the fact that $f(t)/t$ is decreasing  (it follows from $f(0)=0$ and the concavity), and that $\sum_{j,k} \tilde{\phi}_j^k(y) \geq 1$ by (d).
	
	Now we choose $\ell = 3$, $\delta = 1/2$, and
	\[
	f(t) = \frac{2m}{\delta} t^{1-1/m}
	\]
	which allows the existence of the function $\xi$ as required before by a simple cutoff argument applied to $\tilde\xi(t) = \chi_{[0,\infty)}(t)\left(\frac{2t}\delta\right)^m$. From \eqref{eq:aaa} we deduce that
	\[
	K(y) \leq 4\kappa_X(1/49) \leq 4 \lambda_X^6,
	\]
	we obtain also $N = \kappa_X(2/55) \leq \lambda^5$ and we take $m = \log(4 \lambda_X^6)$.
	
	We can now finish the proof by estimating
	\[
	\begin{split}
	\sum_{i,n} \abs{\nabla \phi_i^n} (y) &\leq
	\frac1{\dist(y,X)}\cdot 4 [96(N+1)] m K(y)^{1/m} \\
	&\lesssim \frac1{\dist(y,X)}\cdot \lambda^5 m e^{m/m} =
	\frac1{\dist(y,X)} \lambda^5 \log(\lambda). \qedhere
	\end{split}
	\]
\end{proof}

\section{Random projections}

\label{sec:random-projections}

The following concept has been introduced by \textcite{ohta} and by \textcite{ambrosio-puglisi}. In these articles the authors identify a generalization of a deterministic projection onto a subset, an idea that underlies several extension results. In order to understand the concept let us suppose that $Y=\mathbb{R}^k$ and $X\subset Y$ is a closed convex set. In this case, for every point $y\in Y$ there exists a unique point of $X$ with minimal distance from $y$ and so we have the projection function $P_X: Y \to X$ that is the identity on $X$ and is $1$-Lipschitz on the whole $Y$. This map allows to build a linear Lipschitz extension operator $T:\Lip(X)\to\Lip(Y)$ simply by composition $Tf = f\circ P_X$. Notice that with this definition the Lipschitz constant of $f$ is also preserved. Clearly this kind of construction works only in particular cases, due to topological obstructions. Even in the Euclidean context, the class of subset $X$ that are Lipschitz retractions of the ambient space is very small. In order to overcome this difficulty, we look for non-deterministic maps that share the same features of projections with regard to the possibility of extending functions with the method outlined above. These objects are the so-called random projections, which in some sense are a probabilistic selection of quasi-minimizers of the distance.


\begin{definition}[Random projection]\label{def:random-projection}
Let $X$ be a closed subspace of a metric space $(Y,\dist)$. We say that a map $\mu:Y\to\Prob(X):y\mapsto\mu_y$ is a \emph{random projection} if $\mu_x=\delta_x$ whenever $x\in X$. We say that it is a \emph{Lipschitz random projection} if $\mu\in\Lip\bigl(Y;W_1(X)\bigr)$.
\end{definition}

\begin{theorem}\label{thm:random-projection}
Let $X\subset(Y,\dist)$ be a closed subset with doubling constant $\lambda$. Then there exists a Lipschitz random projection $\mu\in\Lip\bigl(Y;W_1(X)\bigr)$ with
\[
\Lip(\mu) \lesssim \log\lambda.
\]
\end{theorem}

\begin{remark}\label{rem:lipext}
Notice that any Lipschitz random projection $\mu$ gives automatically a bounded linear extension operator $T: \Lip(X, Z) \to \Lip(Y,Z)$ for every Banach space $Z$ in the following way:
\[
(Tf)(y) = \int_X f(x) \d \mu_y (x).
\]
In fact, thanks to \eqref{eqn:w1lip} we have
\[
\abs{(Tf)(y) - (Tf)(y')}
= \abs*{\int_X f(x) \d (\mu_y - \mu_{y'})}
\leq \Lip(f) \Lip(\mu) \dist(y,y').
\]
Therefore the proof of \autoref{thm:random-projection} can be seen as a proof of the existence of a bounded linear extension operator (see \autoref{thm:lipschitz-extension}).
\end{remark}

\begin{proof}
Without loss of generality we can assume that $Y$ is a Banach space, by possibly embedding $Y\subset C_b(Y)$ thanks to the isometric immersion
\[
y \mapsto \dist(\plchldr,y) - \dist(\plchldr,y_0),
\]
where $y_0\in Y$ is a generic fixed point: this is useful because in order to prove that some function $F:Y \in Z$ is $L$-Lipschitz we need only to prove that its slope is bounded by $L$.

Let $\{V_i, \phi_i, x_i\}_i$ be given by \autoref{prop:lipschitz-partition}. Let us then define the random projection
\[
\mu_y = \sum_{i} \phi_i(y) \delta_{x_i} \quad \text{for $y\in Y\setminus X$},
\qquad \mu_y = \delta_y \quad \text{for $y\in X$}.
\]

Given a function $f\in\Lip_1(X)$, for $y\in Y\setminus X$ we can compute the slope
\[
\begin{split}
\abs*{\nabla_y \int_X f(x) \d\mu_y(x)}
&= \abs*{\nabla_y \sum_{i} \phi_i(y) f(x_i)} \\
&= \abs*{\nabla_y \sum_{i} \phi_i(y) [f(x_i)-f(x_{i_0})]} \\
&\leq \sum_{i} \abs*{\nabla_y \phi_i(y)}\cdot\abs{f(x_i)-f(x_{i_0})} \\
&\leq \sum_{i} \abs*{\nabla_y \phi_i(y)}\cdot\dist(x_i,x_{i_0}) ,
\end{split}
\]
where $i_0$ is any fixed index for which $y\in V_{i_0}$. In order for $\abs*{\nabla_y \phi_i(y)}$ to be non-zero, one must have $y\in V_i$, therefore from the properties of the points $x_i$'s we infer that $\dist(x_i,x_{i_0})\lesssim\dist(y,X)$. With this observation we can continue the previous estimate and obtain
\[
\abs*{\nabla_y \int_X f(x) \d\mu_y(x)}
\lesssim \sum_{i} \abs*{\nabla_y \phi_i(y)}\cdot\dist(y,X)
\lesssim \frac{\log\lambda}{\dist(y,X)}\dist(y,X) = \log\lambda .
\]

For points $x\in X$ and $y\in Y\setminus X$ instead we have the estimate
\[
\begin{split}
\abs*{\int_X f(z)\d\mu_y(z) - \int_X f(z)\d\mu_x(z)}
&= \abs*{\sum_i\phi_i(y)[f(x_i)-f(x)]} \\
&\leq \sum_i \phi_i(y) \bigl(
	\abs*{f(x_i)-f(x_{i_0})} + \abs*{f(x_{i_0})-f(x)} \bigr) \\
&\leq \sum_i \phi_i(y) [\dist(x_i,x_{i_0}) + \dist(x_{i_0},x)] \\
&\lesssim \dist(y,X) + \dist(x_{i_0},x) \\
&\leq \dist(y,X) + \dist(x_{i_0},y) + \dist(y,x) \\
&\lesssim \dist(y,x) ,
\end{split}
\]
so that we have a (better) bound on the slope also at the points in $X$. This fact shows that the map $y\mapsto \int_X f\d\mu_y$ has Lipschitz constant less than $\log\lambda$, up to a universal multiplicative constant.

Finally, \autoref{def:w1} of $W_1$ implies that $\Lip(\mu)\lesssim\log\lambda$, indeed.
\end{proof}


We now move on to the corresponding $C^1$ concept of random projection.

\begin{definition}\label{def:regular-random-projection}
Let $X$ be a subset of a Banach space $Y$. We say that a map $\mu:Y\to\Prob(X)$ is a \emph{regular random projection} if the following conditions hold:
\begin{enumerate}
\item\label{it:rrp-1} for every $y\in Y$ the measure $\mu_y$ is concentrated on $B\bigl(y, \eta\dist(y,X)\bigr)$ for some $\eta>0$;
\item\label{it:rrp-2} for all $f \in C(X)$ the map $F(y) = \int_X f(x)\d\mu_y(x)$ is well defined, belongs to $C(Y)\cap C^1(Y\setminus X)$, and there exists $\nu: Y\setminus X\to \Meas(X;Y^*)$ such that
\begin{equation}\label{eq:differential-F}
\d F_y = \int_X f(x)\d\nu_y(x) \qquad \text{for all $y\in Y\setminus X$};
\end{equation}
\item\label{it:rrp-3} for all $y\in Y\setminus X$ the measure $\nu_y$ is concentrated on $B\bigl(y,\eta\dist(y,X)\bigr)$ and its total variation can be estimated with
\[
\normtv{\nu_y}\leq \frac{C_X}{\dist(y,X)}.
\]
\end{enumerate}
\end{definition}

\begin{remark}\label{rmk:rrp-0-average}
With the definition above we have that $\nu_x(X)=0$ for all $x\in Y\setminus X$, since
\[
\nu_y(X)=\int_X 1\d \nu_y=\d \left(\int_X 1 \d \mu_x \right)_y=\d 1_y=0.
\]
\end{remark}

\begin{theorem}[Regular random projection]\label{thm:regular-random-projection}
Let $Y$ be a Banach space whose norm belongs to $C^1(Y\setminus\{0\})$ and let $X\subset Y$ be a closed subset with doubling constant $\lambda$. Then there exists a regular random projection $\mu_y$ whose associated $\nu_y$ has total variation
\[
\normtv{\nu_y} \lesssim \frac{\lambda^4\log\lambda}{\dist(y,X)}.
\]
\end{theorem}

\begin{proof}
Let $\{V_i, \phi_i, x_i\}_i$ be given by \autoref{prop:C1-partition}. Let us then define the random projection
\[
\mu_y = \sum_{i} \phi_i(y) \delta_{x_i} \quad \text{for $y\in Y\setminus X$},
\qquad \mu_y = \delta_y \quad \text{for $y\in X$}.
\]
Property \eqref{it:rrp-1} of \autoref{def:regular-random-projection} follows immediately from \eqref{it:c1part-3} of \autoref{prop:C1-partition}. Let us fix $f\in C(X)$. The function $F(y) = \int_X f(x)\d\mu_y(x)$ is clearly well defined since the measure $\mu_y$ is supported on a finite number of points. Moreover, it is also $C^1(Y\setminus X)$ because the coefficients $\phi_i(y)$ are $C^1$ themselves. Given a point $y\in Y\setminus X$, it is immediate to check that the differential of $F$ at the point $y$ is represented through \eqref{eq:differential-F} by the vector measure
\[
\nu_y = \sum_i \d(\phi_i)_y \delta_{x_i}.
\]
Finallly, \eqref{it:rrp-3} of \autoref{def:regular-random-projection} follows from \eqref{it:c1part-2} of \autoref{prop:C1-partition}.
\end{proof}

\section{Linear extension operators}
\label{sec:linear-extensions}

\subsection{Lipschitz}\label{subsec:lipext}

In this section we state and prove the main result about the extendability of Lipschitz functions. The theorem has already appeared in \cite{lee-naor-05}, but we provide two independent and shorter proofs.

\begin{theorem}\label{thm:lipschitz-extension}
Let $(Y,\dist)$ be a metric space and $X\subset Y$ a closed subset with finite doubling constant $\lambda$; let moreover $Z$ be a Banach space. Then there exists a linear extension operator $T:\Lip(X;Z)\to\Lip(Y;Z)$ such that
\[
\Lip(Tf) \lesssim \log\lambda \Lip(f) \qquad \forall f\in\Lip(X;Z).
\]
\end{theorem}

As already observed in \autoref{rem:lipext}, this result can be obtained already as a direct consequence of \autoref{thm:random-projection}, but we wanted also to provide a self-contained proof that does not require the construction of a partition of unity, but instead exploits the existence of a doubling measure $\meas$ supported on the whole $X$.

\begin{proof}[Direct proof]
Without loss of generality we can assume that $Y$ is a Banach space, by embedding $Y\subset C_b(Y)$ thanks to the isometric immersion
\[
y \mapsto \dist(\plchldr,y) - \dist(\plchldr,y_0),
\]
where $y_0\in Y$ is a fixed point. In particular we can assume that also $X$ is complete by considering its new closure. Let $\meas$ be a doubling measure on $X$, provided for instance by \cite{doubling}. We consider the random projection $\mu:Y\to\Prob(X)$ absolutely continuous with respect to $\meas$ given by
\[
\mu_y = u_y(x) \meas =
	\frac{\phi^m\left(\frac{\dist(y,x)}{\dist(y,X)}\right)}
	{\int_X \phi^m\left(\frac{\dist(y,z)}{\dist(y,X)}\right) \d\meas(z)}
	\meas,
\]
where $\phi\in C^1\bigl([0,\infty);[0,1]\bigr)$ is such that $\phi(t)=1$ for $t\leq2$, $\phi(t)=0$ for $t\geq3$ and $m>0$ is a parameter to be optimized later. Notice that the denominator is non-zero because $\meas$ is doubling. Roughly speaking, this $\mu$ has to be intended as a suitably smoothed version of
\[
\tilde\mu_y = \procond{\meas}{B\bigl(y,3\dist(y,X)\bigr)}.
\]

Given a function $f\in\Lip(X;Z)$, we define its extension $Tf$ by
\[
Tf(y) = \int_X f(x) \d\mu_y(x).
\]
In order to compute $\Lip(Tf)$, we now proceed by estimating the slope of the density $u_y$.

By Leibniz and Fatou\footnote{To apply the latter in order to the pass the slope inside the integral, we need also that
\[
\sup_{\substack{z\in X \\ \dist(y,y')<\frac12\dist(y,X)}}
	\frac1{\dist(y,y')} \abs*{
	\phi^m\left(\frac{\dist(y',z)}{\dist(y',X)}\right)
	- \phi^m\left(\frac{\dist(y,z)}{\dist(y,X)}\right) } < \infty .
\]
} we have
\[
\begin{split}
\abs{\nabla_y u_y(x)}
&\leq \frac{\abs*{\nabla_y \phi^m\left(\frac{\dist(y,x)}{\dist(y,X)}\right)}}
	{\int_X \phi^m\left(\frac{\dist(y,z)}{\dist(y,X)}\right) \d\meas(z)}
	+ \frac{\phi^m\left(\frac{\dist(y,x)}{\dist(y,X)}\right)
		\int_X \abs*{\nabla_y\phi^m\left(\frac{\dist(y,z)}{\dist(y,X)}\right)} \d\meas(z)}
	{\left[\int_X \phi^m\left(\frac{\dist(y,z)}{\dist(y,X)}\right) \d\meas(z)\right]^2}
\end{split}
\]
Integrating in $x$ and simplifying we obtain
\[
\begin{split}
\int_X \abs{\nabla_y u_y(x)} \d\meas(x)
= 2 \frac{\int_X \abs*{\nabla_y\phi^m\left(\frac{\dist(y,z)}{\dist(y,X)}\right)} \d\meas(z)}
	{\int_X \phi^m\left(\frac{\dist(y,z)}{\dist(y,X)}\right) \d\meas(z)}.
\end{split}
\]
One can then compute
\[
\begin{split}
\abs*{\nabla_y\phi^m\oleft(\frac{\dist(y,z)}{\dist(y,X)}\right)}
&\leq m\phi^{m-1}\oleft(\frac{\dist(y,z)}{\dist(y,X)}\right)
	\abs*{\phi'\oleft(\frac{\dist(y,z)}{\dist(y,X)}\right)} \cdot
	\abs*{\nabla_y \left(\frac{\dist(y,z)}{\dist(y,X)}\right)} \\
&\leq m\phi^{m-1}\oleft(\frac{\dist(y,z)}{\dist(y,X)}\right)
	\abs*{\phi'\oleft(\frac{\dist(y,z)}{\dist(y,X)}\right)}
	\frac1{\dist(y,X)} \left(1+\frac{\dist(y,z)}{\dist(y,X)}\right).
\end{split}
\]

Plugging this into the previous equation, observing that the ratio $\frac{\dist(y,z)}{\dist(y,X)}<3$ where $\phi$ is not vanishing and using H\"older inequality in the second step\footnote{With exponents $m/(m-1)$ and $m$.} we get
\[
\begin{split}
\int_X \abs{\nabla_y u_y(x)} \d\meas(x)
&\leq \frac{8m}{\dist(y,X)} \cdot \frac
	{\int_X \phi^{m-1}\left(\frac{\dist(y,z)}{\dist(y,X)}\right)
		\abs*{\phi'\left(\frac{\dist(y,z)}{\dist(y,X)}\right)} \d\meas(z)}
	{\int_X \phi^m\left(\frac{\dist(y,z)}{\dist(y,X)}\right) \d\meas(z)} \\
&\leq \frac{8m}{\dist(y,X)} \left(
	\frac{\int_X \abs*{\phi'\left(\frac{\dist(y,z)}{\dist(y,X)}\right)}^m \d\meas(x)}
	{\int_X \phi^m\left(\frac{\dist(y,z)}{\dist(y,X)}\right) \d\meas(x)}
	\right)^{1/m} \\
&\leq \frac{8m}{\dist(y,X)} \left(
		\frac{\meas\bigl(B\bigl(y,3\dist(y,X)\bigr)\bigr)}
		{\meas\bigl(B\bigl(y,2\dist(y,X)\bigr)\bigr)}
	\right)^{1/m} .
\end{split}
\]
The ratio appearing in the last formula is related to the doubling constant $\lambda$, however one has to be a bit careful because the point $y$ does not belong to $X$. By fixing a point $\tilde y\in X$ such that $\dist(y,\tilde y)\leq(1+\eps)\dist(y,X)$ we get
\[
\frac{\meas\bigl(B\bigl(y,3\dist(y,X)\bigr)\bigr)}
		{\meas\bigl(B\bigl(y,2\dist(y,X)\bigr)\bigr)}
\leq \frac{\meas\bigl(B\bigl(\tilde y,(4+\eps)\dist(y,X)\bigr)\bigr)}
		{\meas\bigl(B\bigl(\tilde y,(1-\eps)\dist(y,X)\bigr)\bigr)}
\leq \lambda^3.
\]
Hence
\[
\int_X \abs{\nabla_y u_y(x)} \d\meas(x)
\lesssim \frac{m\lambda^{3/m}}{\dist(y,X)}
\lesssim \frac{\log\lambda}{\dist(y,X)}
\]
by choosing $m=\frac13\log\lambda$.

We can finally estimate the Lipschitz constant of $Tf$. We start with its slope at $y\in Y\setminus X$. Fixing a point $\tilde y\in X$ such that $\dist(y,\tilde y)\lesssim\dist(y,X)$, we have
\[
\begin{split}
\abs*{\nabla Tf}(y)
&\leq \abs*{\nabla_y \int_X [f(x)-f(\tilde y)] \d\mu_y(x)} \\
&\leq \int_{B\bigl(y,3\dist(y,X)\bigr)} \abs{f(x)-f(\tilde y)}\cdot \abs{\nabla_y u_y(x)} \d\meas(x) \\
&\lesssim \int_{B\bigl(y,3\dist(y,X)\bigr)} \Lip(f)[\dist(x,y)+\dist(y,\tilde y)]\frac{\log\lambda}{\dist(y,X)}\d\mu_y(x) \\
&\lesssim \Lip(f)\dist(y,X)\frac{\log\lambda}{\dist(y,X)} \\
&\lesssim \log\lambda\Lip(f),
\end{split}
\]
where we were able to bring the slope inside the integral because the difference ratios near $y$ are uniformly bounded in $x$. Similarly, for $x\in X$ and $y\in Y\setminus X$ one can compute
\[
\begin{split}
\abs{Tf(y)-Tf(x)}
&\leq \int_X \abs{f(z)-f(x)}\d\mu_y(z) \\
&\leq \Lip(f) \int_{B\bigl(y,3\dist(y,X)\bigr)} \dist(z,x)\d\mu_y(z) \\
&\leq \Lip(f) \int_{B\bigl(y,3\dist(y,X)\bigr)} [\dist(z,y)+\dist(y,x)]\d\mu_y(z) \\
&\lesssim \Lip(f) [\dist(y,X)+\dist(y,x)] \\
&\lesssim \Lip(f)\dist(x,y)
\end{split}
\]
These two computations prove the Lipschitzianity of the map $Tf$, whith constant $\Lip(Tf)\lesssim\log\lambda\Lip(f)$, since the space $Y$ is Banach.
\end{proof}

\begin{remark}
Actually, the previous proof is an alternative self-contained construction of a Lipschitz random projection $\mu$ that does not use a Lipschitz partition of unity.
\end{remark}

\subsection{Whitney}

The goal of this section is to generalize Whitney's extension theorem \cite{whitney} to Banach spaces.

Let $Y$ be a Banach space and let $X\subset Y$ be a closed subset of $X$, we assume that $f:X\to \setR$ and $L:X\to Y^*$ are given functions. We define
\[
R(x,y)=f(y)-f(x)-L_x(y-x) \qquad x,y\in X.
\]
Our aim is to find conditions on $R$ and $X$ in order to have a $C^1$ extension of $f$ at the whole $Y$ and we want that its differential coincides with $L$ in $X$. The classical Whitney's extension theorem ensures that when $Y=\setR^n$ and $R(x,y)=o(\abs{x-y})$ in a suitable sense then the $C^1$ extension there exists. Our result is the following:

\begin{theorem}\label{thm:whitney}
Let $Y$ be a Banach space whose norm belongs to $C^1(Y\setminus\{0\})$ and let $X\subset Y$ be a closed subset with doubling constant $\lambda$. Given two continuous functions $f:X\to\setR$ and $L:X\to Y^*$, define the remainder
\[
R(x,y)=f(y)-f(x)-L_x(y-x) \qquad \text{for $x,y\in X$, $x\neq y$}
\]
and assume that the function
\[
(x,y)\mapsto \frac{R(x,y)}{\abs{y-x}}
\]
can be extended to a continuous function on $X\times X$ that takes the value $0$ where $y=x$. Then there exists an extension $\tilde f\in C^1(Y)$ such that $d\tilde f_x=L_x$ for all $x\in X$.

Moreover, the extension operator $(f,L)\mapsto\tilde f$ is linear.
\end{theorem}

First we prove a key lemma, that is an integral version of $R(x,y)=o(\abs{x-y})$, given our hypotesis on $R$.

\begin{lemma}\label{lem:integral-decay-of-R}
Let $\bar{\mu}:Y\to \Meas_+(X)$ be a weakly measurable map such that $\abs{\bar{\mu}_y}(X)\leq 1$ and there exists $C>0$ such that  $\supp \bar{\mu}_y\in B\bigl(y,C\dist (y,X)\bigr)$ for all $y\in Y$. Assuming the hypothesis of the \autoref{thm:whitney}, for all $x\in X$ we have
\[
\int_X \abs{R(z,x)} \d \bar{\mu}_y(z) = o(\abs{x-y})
\qquad \text{as $y\to x$}.
\]
\end{lemma}

\begin{proof}
Let $\tilde{y}\in X$ be a point such that $\abs{y-\tilde{y}}\leq 2\dist(y,X)$. We can estimate
\begin{align*}
\abs{R(z,x)} &\leq \abs{R(z,x)-R(z,\tilde{y})} + \abs{R(z,\tilde{y})} \\
&= \abs{f(\tilde{y})-f(x)-L_z(\tilde{y}-x)} + \abs{R(z,\tilde{y})} \\
&\leq \abs{f(\tilde{y})-f(x)-L_x(\tilde{y}-x)}
	+ \abs{(L_z-L_x)(\tilde{y}-x)} + \abs{R(z,\tilde{y})} \\
&\leq \abs{R(x,\tilde{y})} + \norm{L_z-L_x}\abs{\tilde{y}-x}
	+ \abs{R(z,\tilde{y})}.
\end{align*}
We observe that
\begin{equation}\label{eq:3ineq}
\abs{\tilde{y}-x} \leq \abs{\tilde{y}-y} + \abs{y-x}
\leq 2\dist(y,X)+\abs{y-x} \leq 3\abs{y-x},
\end{equation}
therefore we have
\[
\int_X \abs{R(z,x)} \d\bar{\mu}_y(z)
\leq \underbrace{\vphantom{\int_X}\abs{R(x,\tilde{y})}}_A
	+ \underbrace{3\abs{y-x}\int_X \abs{L_z-L_x} \d\bar{\mu}_y(z)}_B
	+ \underbrace{\int_X \abs{R(z,\tilde{y})} \d\bar{\mu}_y(z)}_C .
\]
We analize each contribution separately.
\begin{itemize}
\item[($A$)] Using \eqref{eq:3ineq} and the continuity of $(x,y) \mapsto R(x,y)/\abs{x-y}$ we have
\[
\frac{\abs{R(x,\tilde{y})}}{\abs{x-y}} \leq
3 \frac{\abs{R(x,\tilde{y})}}{\abs{x-\tilde{y}}} \to 0.
\]

\item[($B$)] The term $\int_X \abs{L_z-L_x} \d\bar{\mu}_y(z)$ is infinitesimal as $y$ goes to $x$ because the map $z\mapsto\abs{L_z-L_x}$ is continuous and $\supp \bar{\mu}_y\in B\bigl(y,C\dist(y,X)\bigr)$.

\item[($C$)] We can estimate
\begin{align*}
\int_X \abs{R(z,\tilde{y})} \d\bar{\mu}_y(z)
&= \int_{X\cap B\bigl(y, C\dist(y,X)\bigr)}
	\abs{R(z,\tilde{y})} \d\bar{\mu}_y(z) \\
&= \int_{X\cap B\bigl(y, C\dist(y,X)\bigr)} \abs{z-\tilde{y}}
	\frac{\abs{R(z,\tilde{y})}}{\abs{z-\tilde{y}}} \d\bar{\mu}_y(z) \\
&\leq \int_X \abs{\tilde{y}-y}
		\frac{\abs{R(z,\tilde{y})}}{\abs{z-\tilde{y}}} \d\bar{\mu}_y(z) \\
&\hspace{2cm} + \int_{X\cap B\bigl(y, C\dist(y,X)\bigr)} \abs{y-z}
		\frac{\abs{R(z,\tilde{y})}}{\abs{z-\tilde{y}}} \d\bar{\mu}_y(z) \\
&\leq \abs{\tilde{y}-y} \int_X
		\frac{\abs{R(z,\tilde{y})}}{\abs{z-\tilde{y}}} \d\bar{\mu}_y(z) \\
&\hspace{2cm} + C\dist(y,X) \int_X
		\frac{\abs{R(z,\tilde{y})}}{\abs{z-\tilde{y}}} \d\bar{\mu}_y(z) \\
&\leq (2+C) \abs{y-x} \int_X
	\frac{\abs{R(z,\tilde{y})}}{\abs{z-\tilde{y}}}\d \bar{\mu}_y(z) .
\end{align*}
Finally we observe that again using \eqref{eq:3ineq} we have $\tilde{y} \to x$ and thanks to the continuity of $(x,y) \mapsto R(x,y)/\abs{x-y}$ we have
\[
\int_X \frac{\abs{R(z,\tilde{y})}}{\abs{z-\tilde{y}}} \d\bar{\mu}_y(z)
\leq \sup_{z\in B\bigl(y, C\dist(y,X)\bigr) \cap X}
	\frac{\abs{R(z,\tilde{y})}}{\abs{z-\tilde{y}}} \to 0. \qedhere
\]
\end{itemize}
\end{proof}

\begin{proof}[Proof of \autoref{thm:whitney}]
Let $\mu$ be a regular random projection as provided by \autoref{thm:regular-random-projection}. We define the extension of $f$ as
\begin{equation}\label{eq:C1-extension}
\tilde f(y) = \int_X [f(z)+L_z(y-z)] \d\mu_y(z).
\end{equation}

We first prove that the function $\tilde f$ is differentiable at any point $x\in X$ and that $\d f_x=L_x$. Indeed, we have
\[
\begin{split}
\abs{\tilde f(y) - \tilde f(x) - L_x(y-x)}
&= \abs*{\int_X [f(z)+L_z(y-z)]\d \mu_y(z)-f(x)-L_{x}(y-x)} \\
&\leq \abs*{\int_X [f(x)-f(z)-L_z(x-z)]\d\mu_y(z)} \\
&\qquad \qquad \qquad \qquad + \abs*{\int_X(L_z-L_x)(y-x)\d \mu_y(z)} \\
&\leq \int_X \abs{R(z,x)}\d\mu_y(z) + \abs{y-x}\int_X |L_z-L_x|\d \mu_y(z),
\end{split}
\]
the last term is $o(\abs{y-x})$ thanks to \autoref{lem:integral-decay-of-R} and the continuity of $L$.

Now we observe that $\tilde f\in C^1(Y\setminus X)$ and
\[
\d\tilde f_y = \int_X L_z \d\mu_y(z) + \int_X [f(z)+L_z(y-z)] \d\nu_y(z) \qquad \forall y\in Y\setminus X
\]
by a simple differentiation of \eqref{eq:C1-extension} and using \eqref{it:rrp-2} of \autoref{def:regular-random-projection}.

In order to conclude the proof we have to check that $y\mapsto\d\tilde f_y$ is a continuous map from $Y$ to $Y^*$. We already know that the differential of $\tilde f$ is continuous on the open set $Y\setminus X$ and when it is restricted to $X$, therefore it is enough to estimate $\abs{\d\tilde f_y-\d\tilde f_x}$ with $y\in Y\setminus X$ and $x\in X$. Fixing a point $\tilde y\in X$ such that $\abs{y-\tilde y}\leq 2\dist(y,X)$, we have
\[
\begin{split}
\abs{\d\tilde f_y - \d\tilde f_x}
&\leq \abs{\d\tilde f_y - \d\tilde f_{\tilde y}}
	+ \abs{\d\tilde f_{\tilde y} - \d\tilde f_x} \\
&= \abs{\d\tilde f_y - \d\tilde f_{\tilde y}}
	+ \abs{L_{\tilde y} - L_x}.
\end{split}
\]
Now we estimate the first term as
\[
\begin{split}
\abs{\d\tilde f_y - \d\tilde f_{\tilde y}}
&= \abs*{ \left( \int_X L_z \d\mu_y(z)
	+ \int_X [f(z)+L_z(y-z)] \d\nu_y(z)\right)-L_{\tilde{y}} } \\
&\leq \abs*{ \int_X f(z)+L_z(y-z) \d\nu_y(z) }
	+ \abs*{\int_X L_z\ d \mu_y(z)-L_{\tilde{y}} } \\
&\leq \abs*{\int_X f(z)+L_z(y-z) \d\nu_y(z) }
	+ \int_X \abs{L_z-L_{\tilde{y}}} \d\mu_y(z).
\end{split}
\]
Recalling \autoref{rmk:rrp-0-average} we have
\[
\begin{split}
\abs*{\int_X f(z) + L_z(y-z) \d\nu_y(z)}
&= \abs*{\int_X f(z)-f(\tilde{y})-L_z(z-\tilde{y}) \d\nu_y(z)} \\
	&\hspace{2cm}+ \abs*{\int_X L_z(y-\tilde{y}) \d\nu_y(z)} \\
&\leq \! \int _X \! \abs{R(z,\tilde{y})} \d\abs{\nu_y}(z)
	+ \abs{y-\tilde{y}} \int_X \abs{L_z-L_{\tilde{y}}} \d\abs{\nu_y}(z).
\end{split}
\]
Using the property \eqref{it:rrp-3} in \autoref{def:regular-random-projection} we can write $\abs{\nu_y}=\frac{C}{\dist(y,X)}\bar{\mu}_y$ and we notice that $\bar{\mu}_y$ satisfies the hypotesis in \autoref{lem:integral-decay-of-R}. Moreover recalling the assumption $\abs{y-\tilde{y}}\leq 2\dist (y,X)$ we have
\begin{multline*}
\abs*{\int_X f(z)+L_z(y-z) \d\nu_y(z)} \\
\leq \frac{C}{\dist(y,X)}\int_X \abs{R(z,\tilde{y})}\d\bar{\mu}_y(z)
+ \frac{C\abs{y-\tilde{y}}}{\dist(y,X)}\int_X\abs{L_z-L_{\tilde{y}}}
		\d\bar{\mu}_y(z) \\
\leq \frac{2C}{\abs{\tilde{y}-y}}
	\int_X \abs{R(z,\tilde{y})} \d\bar{\mu}_y(z) +
	2C\int_X \abs{L_z-L_{\tilde{y}}} \d\bar{\mu}_y(z).
\end{multline*}
Finally putting all together
\[
\abs{\d\tilde f_y - \d\tilde f_x} \leq
\abs{L_{\tilde y} - L_x} + (2C+1)\int_X \abs{L_z-L_{\tilde{y}}} \d\mu_y(z)
	+ \frac{2C}{\abs{\tilde{y}-y}} \int_X \abs{R(z,\tilde{y})} \d\bar{\mu}_y(z).
\]
Recalling $\abs{x-\tilde{y}}\leq 3\abs{x-y}$ and \autoref{lem:integral-decay-of-R} we conclude that $\abs{\d\tilde f_y - \d\tilde f_x} \to 0$ when $y$ goes to $x$.
This shows that $\d \tilde f$ is continuous also in every point of $X$ and concludes the proof.
\end{proof}

%
%
%


\phantomsection
\addcontentsline{toc}{section}{\refname}
\printbibliography

\end{document}